\documentclass[11pt]{amsart}
\usepackage{amsmath,amsthm,amssymb,latexsym,color}
\usepackage{hyperref}
\usepackage{lipsum}
\usepackage{todonotes}
\usepackage{multirow}
\usepackage[mathscr]{eucal}

\usepackage{algorithmicx,algorithm,algpseudocode}
\algnewcommand{\Input}{\item[\textbf{Input:}]}
\algnewcommand{\Output}{\item[\textbf{Output:}]}

\usepackage{xcolor}
\date{}
\usepackage{pgfplots}
\usepackage{tikz}
\usetikzlibrary{arrows,matrix,positioning}
\usepackage{graphicx}
\usepackage{subfigure}

\title[Convex hull of random points]
{On the probability that the convex hull of random points contains the origin}
\author{
Konstantin Tikhomirov
}
\address{
\medskip
\noindent
Department of Mathematical Sciences\\
Carnegie Mellon University\\
Wean Hall 6113\\
Pittsburgh, PA 15213\\
\texttt{\small
e-mail:   ktikhomi@andrew.cmu.edu}
}

\newtheorem{theorem}{Theorem}[section]
\newtheorem*{theorem*}{Theorem}

\newtheorem{lemma}[theorem]{Lemma}
\newtheorem{cor}[theorem]{Corollary}
\newtheorem{defi}[theorem]{Definition}
\newtheorem{prop}[theorem]{Proposition}

\theoremstyle{definition}
\newtheorem{Remark}[theorem]{Remark}
\newtheorem*{Remark*}{Remark}
\newtheorem*{Problem*}{Problem}

\theoremstyle{plain}

\newcommand{\Z}{\mathbb{Z}}
\newcommand{\R}{\mathbb{R}}

\newcommand{\Exp}{\mathbb{E}}

\newcommand{\Event}{\mathcal{E}}
\def\Prob{{\mathbb P}}

\newcommand{\sign}{\mathbf{Sign}}

\def\col{{\rm col}}

\def\row{{\rm row}}
\def\dist{{\rm dist}}
\def\Net{{\mathcal N}}

\def\LCD{{\rm LCD}}
\def\Comp{{\rm Comp}}
\def\Incomp{{\rm Incomp}}

\def\conv{{\rm conv}}

\begin{document}

\maketitle

\begin{abstract}
The classical theorem of Wendel provides an exact formula for
the probability that the convex hull of independent symmetrically distributed vectors in $\R^d$
contains the origin as long as the distributions of the vectors are continuous.
In this note, we provide an extension to Wendel's theorem for independent random vectors $X_1,\dots,X_n$ with i.i.d
components having a (possibly discrete) symmetric distribution of 
unit variance.
As a related observation, we
give sharp estimates on the probability that a random linear program of the form
``$\max\langle x,{\mathfrak c}\rangle\quad\mbox{subject to }\langle X_i,x\rangle\leq 1,\;i\leq n$'',
is bounded.
\end{abstract}

\section{Introduction}

\subsection{Literature overview and results}
The problem of computing the probability that the convex
hull of a set of random vectors contains the origin has attracted considerable attention;
see, in particular, \cite{Wendel,Eldan,TY17,KVZGAFA,KVZ17,VZ18,KZ20,GKZ21,HLO23}.
The following is a classical theorem of Wendel \cite{Wendel} based on combinatorial properties of hyperplane arrangements:
\begin{theorem}[\cite{Wendel}]\label{akjfblfjhblfjwhbflqjwhbf}
Let $n>d$, and let $X_1,X_2,\dots,X_n$ be independent random vectors in $\R^d$ such that
\begin{itemize}
\item Every $X_i$ is symmetrically distributed i.e $X_i\stackrel{d}{\sim}-X_i$, and
\item With probability one, every $d$--tuple of vectors are linearly independent.
\end{itemize}
Then the probability that the convex hull of $X_1,X_2,\dots,X_n$
contains the origin is given by
$$
p_{n,d}:=1-2^{-n+1}\sum_{k=0}^{d-1}{n-1\choose k}.
$$
\end{theorem}
While the second condition in the above theorem is crucial, it
clearly fails whenever $X_1,\dots,X_n$ are i.i.d having 
discrete distributions. On the other hand,
numerous results within non-asymptotic random matrix theory
show that, up to appropriate normalization, collections of arbitrary random vectors with i.i.d components
share similar characteristics with standard Gaussian vectors.
The purpose of this note is to 
review the phenomenon in the context of Wendel's theorem.
First, we provide a matching statement 
for a class of random vectors with i.i.d components, allowing for discrete distributions:
\begin{theorem}\label{main}
For every $K\geq 1$ there is $c>0$ depending only on $K$ with the following property.
Let $\xi$ be a symmetric random variable of unit variance, and with a subgaussian\footnote{We refer to Preliminaries for the definition of a subgaussian random variable} moment bounded above by $K$.
Let $n>d$, and let $X_1,X_2,\dots,X_n$ be random vectors in $\R^d$
with i.i.d coordinates equidistributed with $\xi$.
Then
\begin{align*}
p_{n,d}\leq
\Prob\big\{\conv\{X_1,\dots,X_n\}\mbox{ contains the origin}\big\}
\leq p_{n,d}
+2\exp(-cd),
\end{align*}
and
\begin{align*}
p_{n,d}-2\exp(-cd)\leq
\Prob\big\{\conv\{X_1,\dots,X_n\}\mbox{ contains the origin in the interior}\big\}
\leq p_{n,d}.
\end{align*}
\end{theorem}

\begin{Remark}
Note that Theorem~\ref{main} identifies $\frac{n}{d}=2\pm O(d^{-1/2})$ as the window
in which transition from ``does not contain the origin w.h.p''
to ``contains the origin w.h.p'' occurs.
\end{Remark}

\begin{Remark}[General position]
We emphasize that 
the above theorem covers certain cases in which the random vectors
$X_1,\dots,X_n$ are {\bf not} in general position with high probability.
As an example, let $n=2d$, and let
$X_1,\dots,X_n$ have i.i.d $\pm1$ (Rademacher)
coordinates. It is not difficult to verify that in such setting
the collection of vectors $\{X_1,\dots,X_n\}$ is not in general position
with probability tending to one as $d\to\infty$,
that is, no trivial reduction to the original Wendel theorem is available.
See Subsection~\ref{subs:overview} below for a related discussion.
\end{Remark}

\medskip

Our second observation deals with random linear programs.
Fix parameters $n$ and $d$, a non-zero cost vector ${\mathfrak c}$, and consider a random linear program of the form
\begin{equation}\label{ljhrbojvhbljhb}
\max\langle x,{\mathfrak c}\rangle\quad\mbox{subject to }Ax\leq {\bf 1},
\end{equation}
where 
${\bf 1}$ is the $n$--dimensional vector of ones,
and $A$ is an $n\times d$ random matrix with i.i.d rows.
Computational complexity of the
shadow simplex method for linear programs of this form were studied by Borgwardt \cite{B82,B87,B99}
(see also \cite{ST04,V09,DH18,HLZ22} for smoothed analysis of 
the shadow simplex method).
Consider the following question: under what conditions on the distributions of the matrix rows
and parameters $d,n$ the program \eqref{ljhrbojvhbljhb}
is bounded w.h.p? In the setting of continuous distributions, \cite{Wendel}
provides a complete answer:
\begin{cor}[Corollary of the result of \cite{Wendel}, see Remark~\ref{aygveuwactuetuc} below]\label{afiyveyiyagiyge}
Let $n\geq d$, and let $A$ be an $n\times d$ random matrix with independent rows such that
\begin{itemize}
\item Every row is symmetrically distributed, and
\item For any fixed hyperplane $H$ and $i\leq n$,
the probability that $\row_i(A)$ belongs to $H$ is zero.
\end{itemize}
Then for any non-zero cost vector ${\mathfrak c}$,
$$
\Prob\big\{\mbox{Linear program \eqref{ljhrbojvhbljhb} is bounded}\big\}=p_{n+1,d}
=1-2^{-n}\sum_{k=0}^{d-1}{n\choose k}.
$$
\end{cor}
As a variation of our main result, we get a matching estimate
for arbitrary symmetric distributions
with i.i.d components:
\begin{theorem}\label{mainLP}
For any $K\geq1$ there is $c>0$ depending only on $K$ with the following property.
Let $n\geq d$, and let $A$ be an $n\times d$ random matrix with
i.i.d symmetrically distributed entries of unit variance and subgaussian moment bounded above by $K$.
Then for any non-random non-zero cost vector $\mathfrak c$,
$$
\bigg|
\Prob\big\{\mbox{Linear program \eqref{ljhrbojvhbljhb} is bounded}\big\}
-1+2^{-n}\sum_{k=0}^{d-1}{n\choose k}
\bigg|\leq 2\exp(-cd).
$$
\end{theorem}
\begin{Remark}
Among standard examples of non-Gaussian distributions satisfying the assumptions of Theorems~\ref{main}
and~\ref{mainLP} are
Rademacher and Bernoulli--Gaussian.
In particular, assume that the entries of $A$
are equidistributed with a Bernoulli--Gaussian variable of the form $b\,g$, where $b$ and $g$
are independent, $b\sim{\rm Bernoulli}(p)$, $g\sim N(0,1)$, and the parameter $p$ is a small
positive constant.
The corresponding optimization problem \eqref{ljhrbojvhbljhb}
can be viewed as a simplest randomized model of a linear program with a diluted coefficient matrix.
\end{Remark}

\subsection{Proof overview}\label{subs:overview}
Theorems~\ref{main} and~\ref{mainLP}
are closely related (see Subsection~\ref{oauebyfiufviqwviyv} for details).
At this point we discuss a ``direct'' proof of Theorem~\ref{main}
since it is less technical than the study of the random linear programs,
and at the same time conveys all the main ideas.

Let $X_1,\dots,X_n$ be i.i.d random vectors in $\R^d$ with i.i.d coordinates having symmetric distribution,
unit variance, and a bounded subgaussian moment.
An approximation argument (see Subsection~\ref{oauebyfiufviqwviyv}) shows that
in order to prove Theorem~\ref{main} it is sufficient to show that
$$
\Prob\big\{\mbox{$\conv\{X_1,\dots,X_{n}\}$ contains the origin on its boundary}\big\}
$$
is exponentially small in $d$.
Disregarding the scenario $\dim\conv\{X_1,\dots,X_{n}\}\leq d-1$
(the probability of that event is shown to be exponentially small in $d$ by applying standard results
of non-asymptotic random matrix theory), the problem is reduced to computing the probability that
${\bf 0}\in\conv\{X_{i_1},\dots,X_{i_d}\}$ for some 
$d$--tuple $X_{i_1},\dots,X_{i_d}$ of affinely independent vertices such that
$\conv\{X_{i_1},\dots,X_{i_d}\}$ lies on the boundary of the polytope (i.e is a facet or a subset
of a facet).
It is not difficult to check that conditioned on the event that ${\bf 0}$ is contained in the affine hyperlane
spanned by $X_{i_1},\dots,X_{i_d}$ and that {\it no proper subset of the vectors
are linearly dependent}, the [conditional] probability of $\big\{{\bf 0}\in\conv\{X_{i_1},\dots,X_{i_d}\}\big\}$
equals $2^{1-d}$. The assumption that $\conv\{X_{i_1},\dots,X_{i_d}\}$ is a subset of a facet also implies that for a vector $y$
orthogonal to $\conv\{X_{i_1},\dots,X_{i_d}\}$, all inner products $\langle y,X_j\rangle$, $j\neq i_1,\dots,i_d$,
are either all non-negative or all non-positive. In an ideal situation when $\Prob\{\langle X_j,y\rangle=0\}=0$,
this would supply an extra factor $2^{-n+d+1}$ to our probability estimate.
Overall, that would imply that for fixed $i_1,\dots,i_d$, the probability that 
$\conv\{X_{i_1},\dots,X_{i_d}\}$ (a) is $(d-1)$--dimensional, (b) is
a subset of a facet and (c) contains the origin, could be estimated from above by
$$
\exp(-c'd)\cdot 2^{1-d}\cdot 2^{-n+d+1},
$$
where the factor $\exp(-c'd)$ is an upper bound on the probability that the affine span of vectors
$X_{i_1},\dots,X_{i_d}$
contains the origin.
It is not difficult to see that
$$
{n\choose d}\,\exp(-c'd)\cdot 2^{1-d}\cdot 2^{-n+d+1}\leq 4\exp(-c'd)
$$
for arbitrary $n>d$, and hence the estimate survives crude union bound over
all unordered $d$--tuples of indices, producing the required result.
The above sketch ignores the fact that
degenerate cases (possible linear dependencies among $X_i$'s) may occur with a positive probability.
The essence of the proof of Theorem~\ref{main} is in carefully accounting for those degenerate cases
and making sure that they do not destroy the probability estimates.

\medskip

On a technical level, the proof builds on the machinery
of non-asymptotic random matrix theory,
with a few adaptations of standard results in that area.
To provide an illustration of the technical aspect,
we state a version of the main lemma used to prove Theorem~\ref{main}.
We refer to Lemma~\ref{alkhjfbkefveutcfujefw} is the text
for the general version of the below statement, necessary
to deal with random linear programs.
Recall that a unit vector $x$ in $\R^m$
is {\it $(\delta,\rho)$--incompressible} for some parameters $\delta,\rho\in(0,1)$
if there is no vector $y$ with at most $\lfloor \delta m\rfloor$
non-zero components
satisfying $\|x-y\|_2\leq \rho$ (see Section~\ref{sec:prelim}
for details).

\begin{lemma}
Fix parameters $K\geq 1$, $M>0$.
There are constants $C,c>0$ and $\rho,\delta\in(0,1)$
depending only on the subgaussian
moment $K$ with the following property.
Let $d\geq m\geq C$, let $\xi$ be a symmetric random variable of unit variance and subgaussian moment bounded
above by $K$, and 
let $X_1,\dots,X_{m}$ be i.i.d random vectors in $\R^d$ with
independent components equidistributed with $\xi$.
Denote by $W$ the $d\times m$
random matrix with columns $X_1,\dots,X_{m}$.
Then the event
\begin{align*}
\Big\{&\mbox{$W$ has rank $m-1$, and there is a $(\delta,\rho)$--incompressible unit vector}\\
&\mbox{$\lambda=(\lambda_1,\dots,\lambda_m)^\top$ with strictly positive components such that
$W\lambda={\bf 0}$}\Big\}
\end{align*}
has probability at most
$$
2^{1-m}
{d\choose m-1}\bigg(\big(C\exp(-c m)\big)^{d-m+1}
+\exp(-M\,m)\bigg).
$$
\end{lemma}

The above lemma can be viewed algebraically as a 
statement about the structure of the kernel
of random rectangular matrices.
For concreteness, let us focus on the case $d=m$.
We note that according to the classical result \cite{RV08},
the probability that the kernel of the matrix $W$ is non-trivial, 
is exponentially small in $m$, with an implicit constant in the power of exponent depending only on $K$. For our purposes, however,
it is crucial to have a much
stronger upper probability estimate of the form
$O\big(2^{-m}\, \exp(-cm)\big)$, with the ``extra'' factor $2^{-m}$
essentially due to the condition that the kernel has a non-empty
intersection with the positive orthant.

\medskip

{\bf Notation.} The spectral norm of a matrix $B$ will be denoted by $\|B\|$.
The unit Euclidean sphere in $\R^h$ is denoted by $S^{h-1}$.
Given a vector $x\in S^{h-1}$, we write $x\chi_{[h-1]}$ for the $(h-1)$--dimensional
vector obtained from $x$ by deleting its $h$--th component.

Throughout the note, for a set $S\subset\R^d$ the {\it interior} of $S$ is the set $S\setminus \partial S$
where $\partial S$ is the boundary of the set defined w.r.t the standard topology in $\R^d$.
Further, the {\it relative interior} of $S$ is the subset $S\setminus \partial S_H$
where $H$ is the affine linear span of $S$ and $\partial S_H$ is the [relative] boundary of $S$
defined w.r.t the standard topology in $H$.

\bigskip

{\bf Acknowledgments.} The work is partially supported by the NSF Grant DMS 2054666. The author would like to thank the Referee for valuable
suggestions on the first version of the paper.

\section{Preliminaries}\label{sec:prelim}

\subsection{Reductions}\label{oauebyfiufviqwviyv}

The first step in proving Theorems~\ref{main} and~\ref{mainLP} is the observation that an approximation argument,
together with Wendel's formula, yield the following one-sided estimates.
\begin{prop}\label{aiygveaiywcvywviey}
Let $n>d$ and let
$Y_1,\dots,Y_{n}$ be independent random vectors in $\R^d$ having symmetric distributions.
Then
\begin{align*}
\Prob&\big\{\mbox{$\conv\{Y_1,\dots,Y_{n}\}$ is $d$--dimensional and contains the origin in its interior}\big\}\\
&\leq 1-2^{-n+1}\sum_{k=0}^{d-1}{n-1\choose k},
\end{align*}
and
\begin{align*}
\Prob\big\{\mbox{$\conv\{Y_1,\dots,Y_{n}\}$ contains the origin}\big\}
\geq 1-2^{-n+1}\sum_{k=0}^{d-1}{n-1\choose k}.
\end{align*}
\end{prop}
We refer to \cite[Proposition~2.12]{KVZGAFA} for a proof of a related statement
which, with minor changes, also verifies Proposition~\ref{aiygveaiywcvywviey}.

As a corollary of the proposition, we get
\begin{prop}\label{aiyevfiyvgiygviygviyg}
Let $n\geq d$, and let $A$ be an $n\times d$ random matrix with independent 
symmetrically distributed rows $X_1,\dots,X_n$.
Then for any non-zero cost vector ${\mathfrak c}$,
\begin{align*}
\bigg|&\Prob\big\{\mbox{Linear program \eqref{ljhrbojvhbljhb} is bounded}\big\}
-1+2^{-n}\sum_{k=0}^{d-1}{n\choose k}
\bigg|\\
&\leq
\Prob\big\{\conv\{X_1,\dots,X_n,{\mathfrak c}\}\;\mbox{contains the origin on its boundary}\big\}.
\end{align*}
\end{prop}
\begin{proof}
Let $s$ be a symmetric sign variable independent from $A$.
Observe that, in view of symmetrical distributions of $X_i$'s,
the linear program
\begin{equation}\label{ahgvfkihgfikchgwckh}
\max\langle x,s\,{\mathfrak c}\rangle\quad\mbox{subject to }Ax\leq {\bf 1},
\end{equation}
is bounded with the same probability as linear program \eqref{ljhrbojvhbljhb}, and
\begin{align*}
&\Prob\big\{\conv\{X_1,\dots,X_n,-s{\mathfrak c}\}\;\mbox{contains the origin on its boundary}\big\}\\
=\,&\Prob\big\{\conv\{X_1,\dots,X_n,{\mathfrak c}\}\;\mbox{contains the origin on its boundary}\big\}.
\end{align*}
Note that
\begin{align*}
\Prob\big\{\mbox{\eqref{ahgvfkihgfikchgwckh} is bounded}\big\}&\geq 
\Prob\big\{\mbox{Polyhedron $\{x\in\R^d:\;Ax\leq{\bf 1},\;\langle x,-s\,{\mathfrak c}\rangle\leq 1\}$ is bounded}\big\}\\
&=\Prob\big\{\conv\{X_1,\dots,X_n,-s\,{\mathfrak c}\}\;\mbox{contains the origin in the interior}\big\},
\end{align*}
and that
$$
\Prob\big\{\mbox{\eqref{ahgvfkihgfikchgwckh} is bounded}\big\}
\leq \Prob\big\{\conv\{X_1,\dots,X_n,-s\,{\mathfrak c}\}\;\mbox{contains the origin}\big\}.
$$
Applying Proposition~\ref{aiygveaiywcvywviey}, we get 
\begin{align*}
\Prob&\big\{\mbox{\eqref{ahgvfkihgfikchgwckh} is bounded}\big\}\\ 
&\geq
1-2^{-n}\sum_{k=0}^{d-1}{n\choose k}
-\Prob\big\{\conv\{X_1,\dots,X_n,-s\,{\mathfrak c}\}\;\mbox{contains the origin on its boundary}\big\},
\end{align*}
and, similarly,
\begin{align*}
\Prob\big\{\mbox{\eqref{ahgvfkihgfikchgwckh} is bounded}\big\}
&\leq
1-2^{-n}\sum_{k=0}^{d-1}{n\choose k}\\
&+
\Prob\big\{\conv\{X_1,\dots,X_n,-s\,{\mathfrak c}\}\;\mbox{contains the origin on its boundary}\big\}.
\end{align*}
The result follows.
\end{proof}
\begin{Remark}\label{aygveuwactuetuc}
Under the assumptions $\Prob\{X_i\in H\}=0$, $1\leq i\leq n$ for every fixed hyperplane $H$,
the last proposition implies that for any non-zero cost vector ${\mathfrak c}$,
$$
\Prob\big\{\mbox{Linear program \eqref{ljhrbojvhbljhb} is bounded}\big\}
=1-2^{-n}\sum_{k=0}^{d-1}{n\choose k},
$$
which yields Corollary~\ref{afiyveyiyagiyge}.
\end{Remark}
\begin{Remark}\label{aouvbowiuhbigiygvyrdx}
Assume that for some $n, d$, and for $X_1,\dots,X_n$ as in Proposition~\ref{aiyevfiyvgiygviygviyg},
we are able to show that
$$
\Prob\big\{\conv\{X_1,\dots,X_n,{\mathfrak c}\}\;\mbox{contains the origin on its boundary}\big\}
\leq 2\exp(-cd)
$$
for $c>0$ depending only on $K$.
In view of Proposition~\ref{aiyevfiyvgiygviygviyg}, this would imply the statement of Theorem~\ref{mainLP} for those $n,d$.
Similarly, in view of Proposition~\ref{aiygveaiywcvywviey},
proving that the probability of the event 
$\big\{\conv\{X_1,\dots,X_n\}\;\mbox{contains the origin on its boundary}\big\}$ is exponentially small in $d$,
would imply Theorem~\ref{main} for those $n,d$.
\end{Remark}

\subsection{Subgaussian variables}

\begin{defi}
Let $K>0$.
A random variable $\xi$ is {\it $K$--subgaussian} if
$$
\Exp\big(\xi^2/K^2\big)\leq 2.
$$
\end{defi}

\begin{lemma}[{Norms of subgaussian matrices;
see, for example, \cite[Section~4.4]{VershyninBook}}]\label{jayucjuatrutqcwquj}
For every $K\geq 1$ and $R>0$ there is $C_{\text{\tiny\ref{jayucjuatrutqcwquj}}}>0$
depending only on $K$ and $R$ with the following property.
Let $A$ be an $N\times n$ random matrix with i.i.d entries of zero mean, unit variance, and with
subgaussian moment bounded above by $K$.
Then
$$
\Prob\big\{\|A\|\geq C_{\text{\tiny\ref{jayucjuatrutqcwquj}}}\sqrt{\max(n,N)}\big\}\leq \exp\big(-R\max(n,N)\big).
$$
\end{lemma}

\subsection{Sparse and compressible vectors}\label{subs:sparsecomp}

\begin{defi}[Sparse vectors]
Given $m\leq h$, a vector $y\in \R^h$ is $m$--sparse if it has at most $m$ non-zero components.
\end{defi}

\begin{defi}[Compressible vectors \cite{LPRT05,RV08}]
Given parameters $\delta,\rho>0$, define $\Comp_h(\delta,\rho)$
as the set of all unit vectors $y$ in $\R^h$ such that the Euclidean distance of $y$ to the set of $\delta h$--sparse
vectors is at most $\rho$.
The vectors from $\Comp_h(\delta,\rho)$ are called $(\delta,\rho)$--compressible.
\end{defi}

\begin{lemma}\label{aljhfakfjhebfkajhbfkahvaghf}
For every $K,R\geq 1$ there are $c_{\text{\tiny\ref{aljhfakfjhebfkajhbfkahvaghf}}},
\beta_{\text{\tiny\ref{aljhfakfjhebfkajhbfkahvaghf}}},\upsilon_{\text{\tiny\ref{aljhfakfjhebfkajhbfkahvaghf}}}>0$
depending only on $K,R$ with the following property.
Let $n,d\geq 1$ satisfy $d\leq Rn$, and let
$B=(B_{ij})$ be a $d\times (n+1)$ matrix,
where the entries $B_{ij}$, $1\leq i\leq d$, $1\leq j\leq n$ are
i.i.d of zero mean, unit variance, and with
subgaussian moment bounded above by $K$,
and the $(n+1)$--st column of $B$ is non-random of Euclidean norm $\sqrt{d}$.
Then with probability at least $1-2\exp(-c_{\text{\tiny\ref{aljhfakfjhebfkajhbfkahvaghf}}}d)$, we have
$$
Bx\neq {\bf 0}\quad\mbox{ for every
$(\beta_{\text{\tiny\ref{aljhfakfjhebfkajhbfkahvaghf}}},\upsilon_{\text{\tiny\ref{aljhfakfjhebfkajhbfkahvaghf}}})$--compressible
vector $x$.}
$$
\end{lemma}
\begin{Remark}
The proof of the above lemma is a standard application of the $\varepsilon$--net argument; see, for example,
proof of \cite[Lemma~3.3]{RV08}, as well as \cite[Section~4.6]{VershyninBook}.
The lemma will be applied with $B=[A^\top, \mathfrak c]$, where $A$ is the coefficient matrix of the linear
program \eqref{ljhrbojvhbljhb}, and $\mathfrak c$ is an appropriately normalized cost vector.
\end{Remark}

\begin{lemma}\label{adslfknaflknkjnkjn}
For every $K\geq 1$ there are $c_{\text{\tiny\ref{adslfknaflknkjnkjn}}},
\delta_{\text{\tiny\ref{adslfknaflknkjnkjn}}},\rho_{\text{\tiny\ref{adslfknaflknkjnkjn}}}>0$
depending only on $K$ with the following property.
Let $\xi$ be a symmetric random variable of unit variance with the subgaussian moment bounded above by $K$.
Let $n\geq d\geq 1$, and let $A$ be an $n\times d$ random matrix with i.i.d
entries equidistributed with $\xi$. Then
\begin{align*}
\Prob\big\{\mbox{There is $y\in \Comp_d(\delta_{\text{\tiny\ref{adslfknaflknkjnkjn}}},\rho_{\text{\tiny\ref{adslfknaflknkjnkjn}}})$
such that $\langle \row_i(A),y\rangle\leq 0$ for all $i\leq n$}\big\}
\leq 2\exp\big(-c_{\text{\tiny\ref{adslfknaflknkjnkjn}}}n\big).
\end{align*}
\end{lemma}
\begin{proof}
We start by observing that there is $c_1\in(0,1)$ depending only on $K$
such that for every $m\geq 1$, every unit vector $\tilde y=(\tilde y_1,\dots,\tilde y_m)\in \R^m$,
and i.i.d variables $\xi_1,\dots,\xi_m$ equidistributed with $\xi$,
we have
$$
\Prob\bigg\{\sum_{i=1}^m \tilde y_i\xi_i\geq\frac{1}{2}\bigg\}\geq c_1
$$
(the claim can be verified by noting that the random variable $\sum_{i=1}^m \tilde y_i\xi_i$
is symmetric, of unit variance, and with subgaussian moment of order $O(K)$; for the last
assertion see, for example, \cite[Section~2.6]{VershyninBook}).
Further, set $\tilde C:=C_{\text{\tiny\ref{jayucjuatrutqcwquj}}}(K,1)$, so that
$$
\Prob\big\{\|A\|> \tilde C\sqrt{n}\big\}\leq \exp(-n).
$$

Next, we choose parameters $c',\delta>0$,
$0<\rho\leq \frac{1}{8\tilde C}$ as functions of $K$ such that for all $n\geq d\geq 1/\delta$,
$$
{n\choose \lfloor 16\tilde C^2\rho^2\,n\rfloor }(1-c_1)^{n-16\tilde C^2\rho^2\,n}\bigg(\frac{3}{\rho}\bigg)^{\lfloor \delta d\rfloor}
{d\choose \lfloor \delta d\rfloor}\leq \exp(-c' n).
$$
In what follows, we assume that $\delta d\geq 1$.
The proof of the lemma
is based on a standard $\varepsilon$--net argument (see, for example, \cite[Section~4.2]{VershyninBook}).
Let $m:=\lfloor \delta d\rfloor$, and for every $m$--subset $I$ of $[d]$ let
$\Net_I$ be a subset of unit vectors in $\R^d$
of size at most $(3/\rho)^m$ such that for every $\tilde 
y$ supported on $I$ and with $1-\rho\leq \|\tilde y\|_2\leq 1$ there is $y'\in\Net_I$ with $\|\tilde y-y'\|_2\leq 2\rho$.
Condition for a moment on any realization of the matrix $A$ such that there is a vector $y\in \Comp_d(\delta,\rho)$
with $\langle \row_i(A),y\rangle\leq 0$ for all $i\leq n$, and let $I\subset[d]$, $|I|=m$, and $y'\in\Net_I$
satisfy $\|y-y'\|_2\leq 2\rho$.
We have
\begin{align*}
\max\big(0,\big\langle \row_i(A),y'\big\rangle\big)^2
\leq
\big|\big\langle \row_i(A),y'-y\big\rangle\big|^2,
\end{align*}
and hence
\begin{align*}
\sum_{i=1}^n \max\big(0,\big\langle \row_i(A),y'\big\rangle\big)^2
\leq \|A(y'-y)\|_2^2\leq 4\rho^2\,\|A\|^2.
\end{align*}
Thus,
$$
\big\langle \row_i(A),y'\big\rangle\geq \frac{1}{2}
$$
for at most
$
16\rho^2\,\|A\|^2
$
indices $i\leq n$.

In view of the above observation, the event
$$
\big\{\mbox{There is $y\in \Comp_d(\delta,\rho)$
such that $\langle \row_i(A),y\rangle\leq 0$ for all $i\leq n$}\big\}
\cap\big\{\|A\|\leq \tilde C\sqrt{n}\big\}
$$
is contained within the event
\begin{align*}
\big\{&\mbox{There are $I\subset[d]$, $|I|=m$, and $y'\in\Net_I$ such that}\\
&\mbox{$\langle \row_i(A),y'\rangle\geq 1/2$
for at most $16\tilde C^2\rho^2\,n$
indices $i\leq n$}\big\}.
\end{align*}
Probability of the latter can be estimated via the union bound argument and our choice of the constants by
$$
{n\choose \lfloor 16\tilde C^2\rho^2\,n\rfloor }(1-c_1)^{n-16\tilde C^2\rho^2\,n}\bigg(\frac{3}{\rho}\bigg)^{m}
{d\choose m}
\leq \exp(-c'n).
$$
Thus, we obtain
\begin{align*}
\Prob\big\{&\mbox{There is $y\in \Comp_d(\delta,\rho)$
such that $\langle \row_i(A),y\rangle\leq 0$ for all $i\leq n$}\big\}\\
&\leq \exp(-c'n)+\Prob\big\{\|A\|> \tilde C\sqrt{n}\big\}\\
&\leq  \exp(-c'n)+\exp(-n).
\end{align*}
Taking $c_{\text{\tiny\ref{adslfknaflknkjnkjn}}}:=\min(c',1)$, we complete the proof.
\end{proof}

\begin{Remark}
Note that the event
$$
\big\{\mbox{There is $y\in \Comp_d(\delta,\rho)$
such that $\langle \row_i(A),y\rangle\leq 0$ for all $i\leq n$}\big\}
$$
can be interpreted as the event that there is a separating hyperplane $H$ for the random polyhedron
$P=\conv\{\row_i(A),\;i\leq n\}$ such that $H$ passes through the origin and the unit normal vector to $H$
is $(\delta,\rho)$--compressible.
\end{Remark}

\subsection{Incompressible vectors and LCD}\label{ajgvefiayewtvfiqyfadsvcsd}

\begin{defi}[Incompressible vectors; see \cite{LPRT05,RV08}]
Given $h\geq 1$ and parameters $\delta,\rho>0$, define
$$\Incomp_h(\delta,\rho):=S^{h-1}\setminus
\Comp_h(\delta,\rho).$$
The vectors from $\Incomp_h(\delta,\rho)$ are called $(\delta,\rho)$--incompressible.
\end{defi}

The following is immediate:
\begin{lemma}\label{akhigvcuyecvuqyqy}
For every $\delta>0$ there is $C_{\text{\tiny\ref{akhigvcuyecvuqyqy}}}(\delta) > 0$ with the following property.
Let $h\geq C_{\text{\tiny\ref{akhigvcuyecvuqyqy}}}$, and assume that 
a vector $x\in S^{h-1}$ is $(\delta,\rho)$--incompressible.
Then $\|x\chi_{[h-1]}\|_2\geq \rho$, and the vector $x\chi_{[h-1]}/\|x\chi_{[h-1]}\|_2$
is $(\delta/2,\rho)$--incompressible.
\end{lemma}
\begin{Remark}
Suppose that $W$ is a $d\times h$ matrix
with the last column non-random (a fixed cost vector ${\mathfrak c}$) and the other columns
with i.i.d random components. To prove certain anti-concentration estimates for $Wx=\sum_{i=1}^h x_i\col_i(W)$
for an incompressible vector $x$,
we will have to consider the ``random part'' $\sum_{i=1}^{h-1} x_i\col_i(W)$,
and properties of the vector $x\chi_{[h-1]}$
(in particular, the incompressibility established in Lemma~\ref{akhigvcuyecvuqyqy}) will play an essential role.
\end{Remark}

\begin{defi}[LCD, \cite{RV09}]
For any $x\in S^{h-1}$ and parameters $\alpha,\gamma>0$, define
$$
\LCD_{\alpha,\gamma}(x):=\inf\big\{\theta>0:\;\dist(\theta x,\Z^h)<\min(\gamma \theta,\alpha)\big\},
$$
where $\Z^h$ is the $h$--dimensional integer lattice, and $\dist(\cdot,\cdot)$
is the Euclidean distance in $\R^h$.
\end{defi}

\begin{lemma}[{\cite[Lemma~3.6]{RV09}}]\label{aflknalkanlknjk}
For every $\delta,\rho\in (0,1)$
there exist $c_{\text{\tiny\ref{aflknalkanlknjk}}}(\delta,\rho) > 0$
and $c'_{\text{\tiny\ref{aflknalkanlknjk}}}(\delta) > 0$
such that the following holds.
Let $x\in S^{h-1}$ be a $(\delta,\rho)$--incompressible vector.
Then, for every $0 < \gamma \leq c_{\text{\tiny\ref{aflknalkanlknjk}}}(\delta, \rho)$
and every $\alpha> 0$, one has
$$
\LCD_{\alpha,\gamma}(x)\geq c'_{\text{\tiny\ref{aflknalkanlknjk}}}(\delta)\sqrt{h}.
$$
\end{lemma}

\begin{theorem}[{A direct corollary of \cite[Theorem~3.3]{RV09}}]\label{akjakjnflkajnlkjnkjn}
For every $K\geq 1$, $\gamma\in(0,1)$ there are
$C_{\text{\tiny\ref{akjakjnflkajnlkjnkjn}}},c_{\text{\tiny\ref{akjakjnflkajnlkjnkjn}}}>0$
depending only on $K,\gamma$ with the following property.
Let $x$ be a unit vector in $\R^m$, and let $\xi_1,\dots,\xi_m$ be i.i.d random variables
of zero mean, unit variance, and subgaussian moment bounded above by $K$.
Then for every $\alpha>0$,
$$
\Prob\bigg\{
\sum_{i=1}^m x_i\xi_i=0
\bigg\}\leq \frac{C_{\text{\tiny\ref{akjakjnflkajnlkjnkjn}}}}{\LCD_{\alpha,\gamma}(x)}+C_{\text{\tiny\ref{akjakjnflkajnlkjnkjn}}}
\exp(-c_{\text{\tiny\ref{akjakjnflkajnlkjnkjn}}}\alpha^2).
$$
\end{theorem}

\bigskip

\noindent {\bf{}For the rest of this section, we fix $K\geq 1$, and
set
$$
\delta:=\min\big(\delta_{\text{\tiny\ref{adslfknaflknkjnkjn}}}(K),\beta_{\text{\tiny\ref{aljhfakfjhebfkajhbfkahvaghf}}}(K,R=3)\big);
\quad \rho:=\min\big(\rho_{\text{\tiny\ref{adslfknaflknkjnkjn}}}(K),
\upsilon_{\text{\tiny\ref{aljhfakfjhebfkajhbfkahvaghf}}}(K,R=3)\big),
$$
where $\delta_{\text{\tiny\ref{adslfknaflknkjnkjn}}}(K)$ and $\rho_{\text{\tiny\ref{adslfknaflknkjnkjn}}}(K)$
are taken from Lemma~\ref{adslfknaflknkjnkjn}, and $\beta_{\text{\tiny\ref{aljhfakfjhebfkajhbfkahvaghf}}}(K,R=3)$,
$\upsilon_{\text{\tiny\ref{aljhfakfjhebfkajhbfkahvaghf}}}(K,R=3)$ --- from Lemma~\ref{aljhfakfjhebfkajhbfkahvaghf}.}

\bigskip

Additionally, set
$$
c_0:=c'_{\text{\tiny\ref{aflknalkanlknjk}}}(\delta/2),\quad \gamma_0:=c_{\text{\tiny\ref{aflknalkanlknjk}}}(\delta/2,\rho),
$$
where $c_{\text{\tiny\ref{aflknalkanlknjk}}}'(\delta)$ and $c_{\text{\tiny\ref{aflknalkanlknjk}}}(\delta,\rho)$ are taken from Lemma~\ref{aflknalkanlknjk}.

\begin{defi}[Level sets, \cite{RV09}]
Given any number $D\geq c_0\sqrt{h}$, define
$$
S_D(\alpha,h):=\big\{x\in \Incomp_h(\delta/2,\rho):\;D\leq \LCD_{\alpha,\gamma_0}(x)<2D\big\}.
$$
We will write $S_D(\alpha)$ when the dimension is clear from the context.
\end{defi}

\begin{lemma}[{Small ball probability for a single vector, \cite[Lemma~4.6]{RV09}}]\label{aitefvigwfvkjbaasfda}
Let $\alpha>0$,  $x\in S_D(\alpha,h-1)$, $v\geq h-1\geq 1$, and let $\tilde B$ be a $v\times (h-1)$ matrix with i.i.d entries
of zero mean, unit variance, and subgaussian moment bounded above by $K$. Then for every $t>0$
and every $z\in\R^v$,
$$
\Prob\big\{\|\tilde B x-z\|_2\leq t\sqrt{v}\big\}\leq \bigg(C_{\text{\tiny\ref{aitefvigwfvkjbaasfda}}}
t+\frac{C_{\text{\tiny\ref{aitefvigwfvkjbaasfda}}}}{D}+C_{\text{\tiny\ref{aitefvigwfvkjbaasfda}}}
\exp\big(-c_{\text{\tiny\ref{aitefvigwfvkjbaasfda}}}\alpha^2\big)\bigg)^{v},
$$
where $C_{\text{\tiny\ref{aitefvigwfvkjbaasfda}}},c_{\text{\tiny\ref{aitefvigwfvkjbaasfda}}}>0$ depend only on $K$.
\end{lemma}
\begin{Remark}
\cite[Lemma~4.6]{RV09} is stated under different assumptions on the matrix dimensions,
however, the proof of the lemma in \cite{RV09} works under our conditions on $v$ and $h$ as well. 
\end{Remark}

The next lemma is based on standard arguments which can already be found in \cite{LPRT05}.
We provide a proof for completeness.
As before, the non-random column of the random matrix in the lemma is 
introduced to deal with the cost vector ${\mathfrak c}$.
\begin{lemma}\label{ajhgcvutyrwvutqwrc}
For every $s>0$ and $M>0$ there is $C_{\text{\tiny\ref{ajhgcvutyrwvutqwrc}}}>0$ depending on $s,M$
with the following property.
Let $v/(1+s)\geq m\geq C_{\text{\tiny\ref{ajhgcvutyrwvutqwrc}}}$,
and let $B=(B_{ij})$ be an $v\times m$ matrix such that
the entries $1\leq i\leq v$, $1\leq j\leq m-1$, are i.i.d
of zero mean, unit variance, and subgaussian moment bounded above by $K$,
and the $m$--th column of $B$ is non-random of Euclidean norm in the interval $[\sqrt{v}/2,2\sqrt{v}]$.
Then
$$
\Prob\big\{B x={\bf 0}\;\mbox{ for some $(\delta,\rho)$--incompressible vector $x$}\big\}
\leq \exp(-M\,v).
$$
\end{lemma}
\begin{proof}
We will assume that $m\geq C_{\text{\tiny\ref{akhigvcuyecvuqyqy}}}(\delta)$,
so that for any $(\delta,\rho)$--incompressible vector $y\in S^{m-1}$, $y\chi_{[m-1]}/\|y\chi_{[m-1]}\|_2$
is $(\delta/2,\rho)$--incompressible.
Without loss of generality, $M\geq 1$.
Set
$$
L:=2C_{\text{\tiny\ref{jayucjuatrutqcwquj}}}(K,2M)+2,
$$
so that, in view of Lemma~\ref{jayucjuatrutqcwquj},
$$
\Prob\big\{\|B\|\geq L\sqrt{v}\big\}\leq \exp\big(-2M\,v\big).
$$
Fix for a moment any $t\in(0,1/2]$, and let $\Net\subset \Incomp_m(\delta,\rho)$
be a $2t$--net on the set of $(\delta,\rho)$--incompressible vectors of size $|\Net|\leq (1+2/t)^m$.
In view of our definition of $c_0$ and Lemma~\ref{aflknalkanlknjk},
for every vector $y\in\Net$ we have $\LCD_{\infty,\gamma_0}\big(
y\chi_{[m-1]}/\|y\chi_{[m-1]}\|_2
\big)\geq c_0\sqrt{m-1}$.
Applying Lemma~\ref{aitefvigwfvkjbaasfda} (with $\tilde B$ obtained from $B$ by removing the $m$--th column),
we have for every $y\in\Net$,
$$
\Prob\big\{\|B y\|_2\leq \tau\sqrt{v}\|y\chi_{[m-1]}\|_2
\big\}\leq \bigg(C_{\text{\tiny\ref{aitefvigwfvkjbaasfda}}}
\tau+\frac{C_{\text{\tiny\ref{aitefvigwfvkjbaasfda}}}}{c_0\sqrt{m-1}}\bigg)^{v},\quad \tau>0,
$$
and hence
\begin{align*}
\Prob&\big\{B x={\bf 0}\;\mbox{ for some $(\delta,\rho)$--incompressible vector $x$, and $\|B\|\leq L\sqrt{v}$}\big\}\\
&\leq \sum\limits_{y\in\Net}
\Prob\big\{\|B y\|_2\leq L\sqrt{v}\cdot 2t\big\}\\
&\leq \big(1+2/t\big)^{v/(1+s)}\,
\bigg(C_{\text{\tiny\ref{aitefvigwfvkjbaasfda}}}\cdot
2tL/\rho+\frac{C_{\text{\tiny\ref{aitefvigwfvkjbaasfda}}}}{c_0\sqrt{m-1}}\bigg)^{v}.
\end{align*}
This implies
\begin{align*}
\Prob&\big\{B x={\bf 0}\;\mbox{ for some $(\delta,\rho)$--incompressible vector $x$}\big\}\\
&\leq \exp\big(-2M\,v\big)
+\big(3/t\big)^{v/(1+s)}\,
\bigg(C_{\text{\tiny\ref{aitefvigwfvkjbaasfda}}}\cdot
2tL/\rho+\frac{C_{\text{\tiny\ref{aitefvigwfvkjbaasfda}}}}{c_0\sqrt{m-1}}\bigg)^{v},\quad t\in(0,1/2].
\end{align*}
Assuming that $m$ is sufficiently large and choosing $t$ such that
$$
\frac{C_{\text{\tiny\ref{aitefvigwfvkjbaasfda}}}\cdot
2tL/\rho+\frac{C_{\text{\tiny\ref{aitefvigwfvkjbaasfda}}}}{c_0\sqrt{m-1}}}
{(3/t)^{1/(1+s)}}\leq \exp(-2M),
$$
we get the result.
\end{proof}

\begin{lemma}[{Cardinality of nets, \cite[Lemma~4.7]{RV09}}]\label{akjehbfikugvwicyqvwiy}
Assume that $\alpha\leq c_0\sqrt{h}$. Then
there exists a $(4\alpha/D)$--net on $S_D(\alpha,h)$
of cardinality at most $(1+9D/\sqrt{h})^h$.
\end{lemma}

As a combination of Lemmas~\ref{aitefvigwfvkjbaasfda} and~\ref{akjehbfikugvwicyqvwiy},
we get the next lemma, which is
a minor modification of \cite[Theorem~4.3]{RV09}:
\begin{lemma}[Null incompressible vectors with small LCD]\label{aljhvflajhfblajhbljh}
For any $M>0$ there are
$\kappa_{\text{\tiny\ref{aljhvflajhfblajhbljh}}}\in(0,c_0]$
and $C_{\text{\tiny\ref{aljhvflajhfblajhbljh}}}\geq C_{\text{\tiny\ref{akhigvcuyecvuqyqy}}},
c_{\text{\tiny\ref{aljhvflajhfblajhbljh}}}>0$
depending only on $K,M$ with the following property.
\begin{itemize}
\item Let $h\geq C_{\text{\tiny\ref{aljhvflajhfblajhbljh}}}$, and let $B=(B_{ij})$ be a $(h-1)\times h$ random matrix,
where the entries $B_{ij}$, $1\leq i\leq h-1$, $1\leq j\leq h-1$, are i.i.d
of zero mean, unit variance, and subgaussian moment bounded above by $K$,
and the $h$--th column of $B$ is non-random. Then the event
\begin{align*}
\big\{&\mbox{There is a $(\delta,\rho)$--incompressible null unit vector $x$}\\
&\mbox{for $B$ with $\LCD_{\kappa_{\text{\tiny\ref{aljhvflajhfblajhbljh}}}\sqrt{h-1},\gamma_0}
\big(x\chi_{[h-1]}/\|x\chi_{[h-1]}\|_2\big)
\leq \exp(c_{\text{\tiny\ref{aljhvflajhfblajhbljh}}}\, h)$}\big\}
\end{align*}
has probability at most $\exp(-M\,h)$.
\item Let $h\geq C_{\text{\tiny\ref{aljhvflajhfblajhbljh}}}$, and let $B$ be a $(h-1)\times h$ random matrix
with i.i.d entries
of zero mean, unit variance, and subgaussian moment bounded above by $K$. Then the event
\begin{align*}
\big\{&\mbox{There is a $(\delta,\rho)$--incompressible null unit vector $x$}\\
&\mbox{for $B$ with $\LCD_{\kappa_{\text{\tiny\ref{aljhvflajhfblajhbljh}}}\sqrt{h},\gamma_0}(x)
\leq \exp(c_{\text{\tiny\ref{aljhvflajhfblajhbljh}}}\, h)$}\big\}
\end{align*}
has probability at most $\exp(-M\,h)$.
\end{itemize}
\end{lemma}
\begin{proof}
We will only prove the first (slightly more technical)
part of the lemma; the proof of the second part is very similar.
In view of Lemma~\ref{akhigvcuyecvuqyqy},
for every $(\delta,\rho)$--incompressible vector $x$, $x\chi_{[h-1]}/\|x\chi_{[h-1]}\|_2$
is $(\delta/2,\rho)$--incompressible.
Denote by $\tilde B$ the $(h-1)\times (h-1)$ submatrix of $B$ obtained by removing the $h$--th column.
Set $L:=C_{\text{\tiny\ref{jayucjuatrutqcwquj}}}(K,4M)$, so that $\tilde B$ satisfies
$$
\Prob\{\|\tilde B\|\geq L\sqrt{h-1}\}\leq\exp(-4M\,(h-1))\leq \exp(-2M\,h).
$$
Fix for a moment any $D\geq c_0\sqrt{h-1}$, $\kappa\in(0,c_0/4]$, set $\alpha:=\kappa\sqrt{h-1}$,
and consider the event
\begin{align*}
\Event_D:=&\big\{\mbox{There is $x\in\R^h$ with
$x\chi_{[h-1]}\in S_D(\alpha)$
with $Bx={\bf 0}$}\big\}\\
&\cap\{\|\tilde B\|\leq L\sqrt{h-1}\}.
\end{align*}
Let $\Net_D$ be a $(4\alpha/D)$--net on $S_D(\alpha)$
of cardinality at most $(1+9D/\sqrt{h-1})^{h-1}$ (the net exists according to Lemma~\ref{akjehbfikugvwicyqvwiy}),
and define a discrete subset $\Net$ of $\R$ as follows:
\begin{itemize}
\item If $\col_h(B)={\bf0}$ then $\Net:=\{0\}$;
\item Otherwise, $\Net:=\big\{\frac{4\alpha L\sqrt{h-1}}{D\|\col_h(B)\|_2}\cdot j\big\}_{j\in\Z,\,|j|\leq 
\frac{D}{4\alpha}}$.
\end{itemize}
Let $\Net_D':=\Net_D\times \Net$, so that
$$
|\Net_D'|\leq (1+9D/\sqrt{h-1})^{h-1}\cdot \Big(\frac{D}{2\alpha}+1\Big).
$$
We claim that
$$
\Event_D\subset\big\{\mbox{There is $y'\in \Net_D'$ with $\|By'\|_2\leq (8\alpha/D)L\sqrt{h-1}$}\big\}.
$$
Indeed, fix any realization of $B$ from $\Event_D$, and let $x$ be a vector from the definition of $\Event_D$.
Since $\|\tilde B\|\leq L\sqrt{h-1}$ and $\|x\chi_{[h-1]}\|_2=1$, we have
$\|\col_h(B)\|_2\,|x_h|\leq L\sqrt{h-1}$.
Define a real number $u_x$ as
\begin{itemize}
\item $u_x:=0$ if $\col_h(B)={\bf 0}$ or $x_h=0$;
\item $u_x:=\sign(x_h)\cdot\frac{4\alpha L\sqrt{h-1}}{D\|\col_h(B)\|_2}\cdot
\big\lfloor\frac{\|\col_h(B)\|_2\,|x_h|}{L\sqrt{h-1}}\cdot \frac{D}{4\alpha}\big\rfloor$, otherwise,
\end{itemize} 
and observe that $u_x\in\Net$ and that
$$\big|\|\col_h(B)\|_2\,x_h-\|\col_h(B)\|_2\,u_x\big|\leq 4\alpha L\sqrt{h-1}/D.$$
Further, let $y$ be an element of $\Net_D$
with $\|x\chi_{[h-1]}-y\|_2\leq 4\alpha/D$, and define $y':=(y,u_x)\in\Net_D'$.
Then
$$
\|By'\|_2\leq \|\tilde B(y-x)\|_2+\|\col_h(B)\|_2\,|u_x-x_h|
\leq 8\alpha L\sqrt{h-1}/D,
$$
and the claim follows.

\medskip

To estimate the probability of $\Event_D$, we apply the claim together with Lemma~\ref{aitefvigwfvkjbaasfda}.
Specifically,
\begin{align*}
\Prob&(\Event_D)\leq \sum_{y'\in\Net_D'}\Prob\big\{\|By'\|_2\leq (8\kappa\sqrt{h-1}/D)L\sqrt{h-1}\big\}\\
&\leq \bigg(\frac{8C_{\text{\tiny\ref{aitefvigwfvkjbaasfda}}}\kappa\sqrt{h-1} L}{D}
+\frac{C_{\text{\tiny\ref{aitefvigwfvkjbaasfda}}}}{D}+C_{\text{\tiny\ref{aitefvigwfvkjbaasfda}}}
\exp(-c_{\text{\tiny\ref{aitefvigwfvkjbaasfda}}}\kappa^2 (h-1))\bigg)^{h-1}\bigg(1+\frac{9D}{\sqrt{h-1}}\bigg)^{h-1}
\Big(\frac{D}{2\alpha}+1\Big).
\end{align*}
Note that for $\kappa$ sufficiently small, for $D\leq \exp(c_{\text{\tiny\ref{aitefvigwfvkjbaasfda}}}\kappa^2 (h-1))$
and for $h$
greater than a large enough constant
(depending on $\kappa,C_{\text{\tiny\ref{aitefvigwfvkjbaasfda}}}, M$), the last quantity
can be made smaller than $\exp(-2Mh)$.
Taking the union bound over the events $\Event_{2^k\,c_0\sqrt{h-1}}$,
$0\leq k\leq \log_2\frac{\exp(c_{\text{\tiny\ref{aitefvigwfvkjbaasfda}}}\kappa^2 (h-1))}{c_0\sqrt{h-1}}$,
we get
\begin{align*}
\Prob\big\{&\mbox{There is a $(\delta,\rho)$--incompressible null unit vector $x$
for $B$}\\
&\mbox{with $\LCD_{\alpha,\gamma_0}(x\chi_{[h-1]}/\|x\chi_{[h-1]}\|_2)
\leq \exp(c_{\text{\tiny\ref{aitefvigwfvkjbaasfda}}}\kappa^2 (h-1))$}\big\}\\
&\hspace{3cm}
\leq \exp(-2Mh)\cdot \bigg(1+\log_2\frac{\exp(c_{\text{\tiny\ref{aitefvigwfvkjbaasfda}}}\kappa^2 (h-1))}{c_0\sqrt{h-1}}\bigg)
+\exp(-2Mh),
\end{align*}
and the result follows.
\end{proof}

\section{Proof of Theorems~\ref{main} and~\ref{mainLP}}

In this section, we assume the same choice for parameters $\delta,\rho$ and $c_0,\gamma_0$
as in Subsection~\ref{ajgvefiayewtvfiqyfadsvcsd}.
\begin{lemma}\label{alkhjfbkefveutcfujefw}
Fix $K\geq 1$, $M>0$, let $C_{\text{\tiny\ref{aljhvflajhfblajhbljh}}},c_{\text{\tiny\ref{aljhvflajhfblajhbljh}}}>0$,
$\kappa=\kappa_{\text{\tiny\ref{aljhvflajhfblajhbljh}}}$
be as in Lemma~\ref{aljhvflajhfblajhbljh}, and let $C_{\text{\tiny\ref{akjakjnflkajnlkjnkjn}}},
c_{\text{\tiny\ref{akjakjnflkajnlkjnkjn}}}>0$ be as in Theorem~\ref{akjakjnflkajnlkjnkjn} (with parameters $K$ and $\gamma_0$).
Assume that $d\geq m\geq C_{\text{\tiny\ref{aljhvflajhfblajhbljh}}}$.
Let $\xi$ be a symmetric random variable of unit variance, and subgaussian moment bounded
above by $K$.
Further, let $X_1,\dots,X_{m-1}$ be i.i.d random vectors in $\R^d$ with
independent components equidistributed with $\xi$, let
${\bf X}$ be a random vector in $\R^d$ of Euclidean norm in the range $[\sqrt{d}/2,2\sqrt{d}]$,
independent from $\{X_1,\dots,X_{m-1}\}$,
and denote by $W$ the $d\times m$
random matrix with columns $X_1,\dots,X_{m-1}$, ${\bf X}$.
Then the event
\begin{align*}
\Big\{&\mbox{$W$ has rank $m-1$, and there is a $(\delta,\rho)$--incompressible unit vector}\\
&\mbox{$\lambda=(\lambda_1,\dots,\lambda_m)^\top$ with strictly positive components such that
$W\lambda={\bf 0}$}\Big\}
\end{align*}
has probability at most
$$
2^{1-m}
{d\choose m-1}\bigg(\Big(\frac{C_{\text{\tiny\ref{akjakjnflkajnlkjnkjn}}}}
{\exp(c_{\text{\tiny\ref{aljhvflajhfblajhbljh}}}\, m)}
+C_{\text{\tiny\ref{akjakjnflkajnlkjnkjn}}}\exp(-c_{\text{\tiny\ref{akjakjnflkajnlkjnkjn}}}\kappa^2 (m-1))\Big)^{d-m+1}
+\exp(-M\,m)\bigg).
$$
\end{lemma}
\begin{Remark}
The above lemma is formulated so as to deal with both the case when ${\bf X}$
is the non-random cost vector and when ${\bf X}$ is one of the random rows of the coefficient matrix $A$.
For that reason, we allow ${\bf X}$ to be random but do not place restrictions on its distribution.
\end{Remark}
\begin{proof}
Without loss of generality, we can assume that the vector ${\bf X}$ is constant.
Denote the event in question by $\Event$.
First, consider an auxiliary event
\begin{align*}
\Event':=\Big\{&\mbox{$W$ has rank $m-1$, and there is a $(\delta,\rho)$--incompressible unit vector}\\
&\mbox{$\lambda=(\lambda_1,\dots,\lambda_m)^\top$ with non-zero components such that
$W\lambda={\bf 0}$}\Big\}.
\end{align*}
Assume that the probability of $\Event'$ is non-zero.
We can construct a mapping $Q:\Event'\to \{-1,1\}^{m}$, which maps
every point $\omega\in\Event'$ to vector
$$
\big(\sign(\lambda_1),\sign(\lambda_2),\dots,\sign(\lambda_{m})=1\big),
$$
where $\lambda$ is the $(\delta,\rho)$--incompressible unit vector with
$\lambda_{m}>0$ satisfying $W\lambda={\bf 0}$
(note that at every point of $\Event'$ the vector is uniquely determined because of the assumptions
on the matrix rank).
Conditioned on $\Event'$, $Q$ assigns equal probabilities to all of the $2^{m-1}$ admissible $\pm 1$ vectors,
which follows from the assumption that $X_i$'s are symmetrically distributed.
Thus,
$$
\Prob\big(\Event\;|\;\Event'\big)=\frac{\Prob(\Event)}{\Prob(\Event')}=2^{1-m},
$$
and it is sufficient to estimate the probability of $\Event'$.

Set $\alpha:=\kappa_{\text{\tiny\ref{aljhvflajhfblajhbljh}}}\sqrt{m-1}$.
For every subset $J\subset[d]$ of size $m-1$, let $W_J$ be the $(m-1)\times m$
submatrix of $W$ obtained by selecting rows with indices in $J$, and let
$\Event_J$ and $\tilde\Event_J$
be the events
\begin{align*}
\Event_J:=\Big\{&\mbox{$W_J$ has rank $m-1$, and there is a $(\delta,\rho)$--incompressible vector
$\lambda=(\lambda_1,\dots,\lambda_m)^\top$}\\
&\mbox{with $\LCD_{\alpha,\gamma_0}
\big(\lambda\chi_{[m-1]}/\|\lambda\chi_{[m-1]}\|_2\big)\geq \exp(c_{\text{\tiny\ref{aljhvflajhfblajhbljh}}}\, m)$
such that
$W\lambda={\bf 0}$}\Big\};\\
\tilde\Event_J:=\Big\{&\mbox{There is a $(\delta,\rho)$--incompressible unit
vector $\lambda=(\lambda_1,\dots,\lambda_m)^\top$}\\
&\mbox{with $\LCD_{\alpha,\gamma_0}\big(\lambda\chi_{[m-1]}/\|\lambda\chi_{[m-1]}\|_2\big)
< \exp(c_{\text{\tiny\ref{aljhvflajhfblajhbljh}}}\, m)$
such that
$W_J\lambda={\bf 0}$}\Big\}.
\end{align*}
First, we observe that
$$
\Event'\subset\bigcup\limits_{|J|=m-1}\Event_J\;\cup\;\bigcup\limits_{|J|=m-1}\tilde\Event_J.
$$
Further, note that in view of Lemma~\ref{aljhvflajhfblajhbljh}, the probability of each $\tilde\Event_J$
is bounded above by $\exp(-M\,m)$.
Next, conditioned on any realization of $W_J$ of rank $m-1$,
a unit vector $\lambda$ such that $W_J\lambda={\bf0}$ is uniquely determined up to a sign,
allowing to decouple the event $\{W_J\lambda={\bf0}\}$
with the events $\{\langle\row_i(W),\lambda\rangle=0\}$, $i\notin J$.
Specifically, we can write for each $J\subset[d]$ with $|J|=m-1$:
\begin{align*}
\Prob(\Event_J)
\leq \Prob\Big\{&\mbox{$W_J$ has rank $m-1$, and there is a $(\delta,\rho)$--incompressible vector
$\lambda=(\lambda_1,\dots,\lambda_m)^\top$}\\
&\mbox{with $\LCD_{\alpha,\gamma_0}
\big(\lambda\chi_{[m-1]}/\|\lambda\chi_{[m-1]}\|_2\big)
\geq \exp(c_{\text{\tiny\ref{aljhvflajhfblajhbljh}}}\, m)$
such that
$W_J\lambda={\bf 0}$}\Big\}\\
&\hspace{-0.5cm}\cdot \sup\limits_y \prod\limits_{i\notin J}\Prob\big\{\langle\row_i(W),y\rangle=0\big\},
\end{align*}
where the supremum is taken over all unit $(\delta,\rho)$--incompressible vectors $y$
with
$$\LCD_{\alpha,\gamma_0}\big(y\chi_{[m-1]}/\|y\chi_{[m-1]}\|_2\big)
\geq \exp(c_{\text{\tiny\ref{aljhvflajhfblajhbljh}}}\, m).$$
Applying Theorem~\ref{akjakjnflkajnlkjnkjn}, we get
$$
\Prob(\Event_J)\leq \bigg(\frac{C_{\text{\tiny\ref{akjakjnflkajnlkjnkjn}}}}
{\exp(c_{\text{\tiny\ref{aljhvflajhfblajhbljh}}}\, m)}+C_{\text{\tiny\ref{akjakjnflkajnlkjnkjn}}}\exp(-c_{\text{\tiny\ref{akjakjnflkajnlkjnkjn}}}\kappa^2 (m-1))\bigg)^{d-m+1}.
$$
Thus, 
$$
\Prob(\Event')\leq {d\choose m-1}\bigg(\Big(\frac{C_{\text{\tiny\ref{akjakjnflkajnlkjnkjn}}}}
{\exp(c_{\text{\tiny\ref{aljhvflajhfblajhbljh}}}\, m)}
+C_{\text{\tiny\ref{akjakjnflkajnlkjnkjn}}}\exp(-c_{\text{\tiny\ref{akjakjnflkajnlkjnkjn}}}\kappa^2 (m-1))\Big)^{d-m+1}
+\exp(-M\,m)\bigg),
$$
and the result follows.
\end{proof}

\begin{proof}[Proof of Theorems~\ref{main} and~\ref{mainLP}]
Note that in order to prove the theorems, it is sufficient to study the regime when
$d$ is large and $1.5d\leq n\leq 3d$
(instead of $1.5$ and $3$ we could pick any other constants strictly less (resp., strictly bigger) than $2$).
Let $\xi$ be a symmetric random variable of unit variance, and subgaussian moment bounded
above by $K$,
and let $X_1,\dots,X_n$ be i.i.d random vectors with independent components equidistributed with $\xi$.
Further, let $X_{n+1}:=\mathfrak c\in\R^d$ be a fixed cost vector of Euclidean norm $\sqrt{d}$. 
In view of Remark~\ref{aouvbowiuhbigiygvyrdx},
it is sufficient to show that
$$
\Prob\big\{\conv\{X_1,\dots,X_n,X_{n+1}\}\;\mbox{contains the origin on its boundary}\big\}
\leq 2\exp(-c'd)
$$
for some $c'>0$ depending only on $K$.
Denote $P_d':=\conv\{X_1,\dots,X_n,X_{n+1}\}$.
Set $M:=5$ and $R:=3$.
We will assume that the parameters $\delta,\rho,\gamma_0>0$ are as in the previous section,
that $\alpha=\kappa_{\text{\tiny\ref{aljhvflajhfblajhbljh}}}\sqrt{d}$
and $C_{\text{\tiny\ref{aljhvflajhfblajhbljh}}},c_{\text{\tiny\ref{aljhvflajhfblajhbljh}}}$
are as in Lemma~\ref{aljhvflajhfblajhbljh} (with our choice of $M$).
We will further assume that
$d\geq \max(C_{\text{\tiny\ref{aljhvflajhfblajhbljh}}},C_{\text{\tiny\ref{akhigvcuyecvuqyqy}}})$ and $\delta n\geq
\max(\sqrt{d},C_{\text{\tiny\ref{aljhvflajhfblajhbljh}}})$.
We claim that
\begin{align*}
\Prob&\big\{\mbox{$P_d'$ contains the origin on its boundary}\big\}\\
&\hspace{1cm}\leq \Prob(\Event_1)
+\Prob(\Event_2)
+\Prob(\Event_3)
+\Prob(\Event_4)
+\Prob(\Event_5)
+\sum_{J\subset I\subset[n+1]:\,\sqrt{d}\leq |J|\leq d,\;|I|=d}\Prob(\Event_{J,I}\setminus (\Event_4\cup\Event_5)),
\end{align*}
where
\begin{align*}
\Event_1&:=\big\{\mbox{$P_d'$ has dimension at most $d-1$ and contains the origin}\big\},\\
\Event_2&:=
\big\{
\mbox{
There is a $(\delta,\rho)$--compressible unit vector}\\
&\hspace{0.9cm}\mbox{$y\in\R^d$ such that $\langle X_i,y\rangle\leq 0$ for all $i\leq n+1$}
\big\},\\
\Event_3&:=\big\{
\mbox{There are $i_1<\dots<i_d$ so that 
$\conv\{X_{i_1},\dots,X_{i_d}\}$ 
contains the origin,}\\
&\hspace{0.9cm}\mbox{and
a $(\delta,\rho)$--incompressible vector $y$ orthogonal to $X_{i_1},\dots,X_{i_d}$,}\\
&\hspace{0.9cm}\mbox{satisfies
$\LCD_{\alpha,\gamma_0}(y)\leq \exp(c_{\text{\tiny\ref{aljhvflajhfblajhbljh}}}\, d)$}
\big\},\\
\Event_4&:=\Big\{
\mbox{There is a $(\delta,\rho)$--compressible vector $z=(z_1,\dots,z_{n+1})\in S^{n}$ such that
$\sum\limits_{i=1}^{n+1} z_i\,X_i={\bf 0}$}
\Big\},\\
\Event_5&:=\big\{\mbox{$\|X_i\|_2\notin[\sqrt{d}/2,2\sqrt{d}]$ for some $1\leq i\leq n$}\big\},
\end{align*}
and for every $J\subset I\subset[n+1]$ with $\sqrt{d}\leq |J|\leq d$ and $|I|=d$,
\begin{align*}
\Event_{J,I}:=\big\{&\mbox{$X_i,\;i\in I$ are affinely independent,
$\conv\{X_i,\;i\in I\}$ is on the boundary of $P_d'$;}\\
&\mbox{$\conv\{X_i,\;i\in J\}$ contains the origin in its relative interior,}\\
&\mbox{and a unit vector $y$ orthogonal to $X_i,\;i\in I$, satisfies $\LCD_{\alpha,\gamma_0}(y)
\geq \exp(c_{\text{\tiny\ref{aljhvflajhfblajhbljh}}}\, d)$}
\big\}.
\end{align*}
Indeed, consider any point $\omega$ of the probability space
at the intersection of $(\Event_1\cup\Event_2\cup\Event_4)^c$ and the event
that $P_d'$ contains the origin on its boundary.
Necessarily, at $\omega$ there are $d$ affinely independent vertices $X_{i_1},\dots,X_{i_d}$
so that $\conv\{X_{i_1},\dots,X_{i_d}\}$ is part of the boundary of $P_d'$
and contains ${\bf 0}$.
Let $y$ be a unit vector orthogonal to $X_{i_1},\dots,X_{i_d}$,
and note that, being in the complement of $\Event_2$,
we get that $y$ is $(\delta,\rho)$--incompressible. If $\LCD_{\alpha,\gamma_0}(y)
\leq \exp(c_{\text{\tiny\ref{aljhvflajhfblajhbljh}}}\, d)$
then $\omega\in\Event_3$. Otherwise, set $I:=\{i_1,\dots,i_d\}$, and let $J$ be a subset of $I$
such that ${\bf 0}$ is contained in the relative interior of $\conv\{X_i,\;i\in J\}$ (note that $J$
exists and has size more than $\sqrt{d}$, since $\omega\in\Event_4^c$).
Then, by the above definition, $\omega\in \Event_{J,I}$, and the claim is verified.
Next, we estimate probabilities of the above events.

\medskip

The probability of $\Event_1$
can be bounded above by the probability of the event that the smallest singular
value of the $d\times n$ matrix with columns $X_i$, $1\leq i\leq n-1$, $X_n-X_{n+1}$, is zero.
The matrix can be viewed as a non-random shift
(with the shift matrix of spectral norm $\sqrt{d}$) of 
the centered $d\times n$ matrix with i.i.d subgaussian entries of unit variance.
Invertibility of such matrices has been 
extensively studied in the literature; for our purposes it suffices to apply the result of \cite{PZ10}
to infer
$$
\Prob(\Event_1)\leq 2\exp(-c_1n),
$$
where $c_1>0$ may only depend on $K$.

\medskip

Further, applying Lemma~\ref{adslfknaflknkjnkjn}, we get
\begin{align*}
\Prob(\Event_2)\leq \Prob\big\{&\mbox{There is a $(\delta,\rho)$--compressible unit vector}\\
&\mbox{$y\in\R^d$ such that $\langle X_i,y\rangle\leq 0$ for all $i\leq n$}\big\} 
\leq 2\exp(-c_{\text{\tiny\ref{adslfknaflknkjnkjn}}} n).
\end{align*}

\medskip

To estimate the probability of $\Event_3$, we apply the second part of
Lemma~\ref{aljhvflajhfblajhbljh} and a union bound argument.
Observe that, in view of equidistribution of $X_1,\dots,X_n$,
\begin{align*}
\Prob\big(\Event_3\big)\leq {n\choose d-1}\Prob\big\{
&\mbox{There is a $(\delta,\rho)$--incompressible vector $y$}\\
&\mbox{orthogonal to $X_1,\dots,X_{d-1}$, with
$\LCD_{\alpha,\gamma_0}(y)\leq \exp(c_{\text{\tiny\ref{aljhvflajhfblajhbljh}}}\, d)$}
\big\},
\end{align*}
and hence, by Lemma~\ref{aljhvflajhfblajhbljh} and our choice of $M$,
$\Prob\big(\Event_3\big)\leq 2^{3d} \exp(-Md)< \exp(-d)$.

\medskip

By our choice of parameters, $\Prob(\Event_4)\leq 2\exp(-c_{\text{\tiny\ref{aljhfakfjhebfkajhbfkahvaghf}}}n)$,
by Lemma~\ref{aljhfakfjhebfkajhbfkahvaghf}.

\medskip

A standard concentration estimate for the Euclidean norm of random vectors with i.i.d subgaussian
components of unit variance (see, for example, \cite[Section~3.1]{VershyninBook}) implies
$$
\Prob(\Event_5)\leq 2\exp(-c_5d)
$$
for some $c_5>0$ depending only on $K$.

\medskip

It remains to estimate the probabilities of the events $\Event_{J,I}\setminus (\Event_4\cup\Event_5)$.
Let $I\subset[n+1]$ be any fixed $d$--subset, and let $J\subset I$ satisfy $|J|\geq \sqrt{d}$.
Denote by $W_J$ the random $d\times |J|$ matrix with columns $X_i,\;i\in J$
(we will arrange the column vectors according to their indices), and let $i_{\max}$ be the largest element of $J$.
Note that the matrix $W=W_J$ satisfies the assumptions of Lemma~\ref{alkhjfbkefveutcfujefw}
with ${\bf X}=X_{i_{\max}}$
(our intention here is to treat both cases $i_{\max}=n+1$, when $X_{i_{\max}}$ is constant,
and $i_{\max}\neq n+1$, when $X_{i_{\max}}$ is equidistributed with the rest of the vectors $X_i,\;i\in J$,
in a uniform way, without splitting the proof into subcases).
Observe that
\begin{align*}
\Event_{J,I}\setminus (\Event_4\cup\Event_5)\subset
\big\{&\mbox{$W_J$ has rank $|J|-1$, the Euclidean norm
of $X_{i_{\max}}$ is in the range}\\
&\mbox{$[\sqrt{d}/2,2\sqrt{d}]$,
and there is a $(\delta,\rho)$--incompressible unit vector}\\
&\mbox{$\lambda=(\lambda_1,\dots,\lambda_{|J|})^\top$ with strictly positive components such that
$W_J\lambda={\bf 0}$}\big\}\\
\cap\,
\big\{
&\mbox{$X_i,\;i\in I$ are affinely independent, $\conv\{X_i,\;i\in I\}$ is on the boundary}\\
&\mbox{of $P_d'$, and a unit vector $y$ orthogonal to $X_i,\;i\in I$,}\\
&\mbox{satisfies $\LCD_{\alpha,\gamma_0}(y)
\geq \exp(c_{\text{\tiny\ref{aljhvflajhfblajhbljh}}}\, d)$}
\big\}.
\end{align*}
Note that the vector $y$ in the description of the last event is uniquely determined by $X_i,\;i\in I$
up to a sign, and that the requirement that $\conv\{X_i,\;i\in I\}$ is [a subset of] a facet of $P_d'$
implies that the inner products $\langle y,X_i\rangle$, $i\notin I$, are either all non-positive
or all non-negative. The $\sigma(X_i,\;i\in I)$--measurability of $\{\pm y\}$ then allows to 
condition on any admissible realization of $X_i,\;i\in I$, and $\pm y$ and use the conditional independence
and identical distribution
of $\langle y,X_i\rangle$, $i\notin I\cup\{n+1\}$, to get, in view of Lemma~\ref{alkhjfbkefveutcfujefw},
\begin{align*}
\Prob&\big(\Event_{J,I}\setminus (\Event_4\cup\Event_5)\big)\\
&\leq 2^{1-|J|}
{d\choose |J|-1}\bigg(\Big(\frac{C_{\text{\tiny\ref{akjakjnflkajnlkjnkjn}}}}
{\exp(c_{\text{\tiny\ref{aljhvflajhfblajhbljh}}} |J|)}
+C_{\text{\tiny\ref{akjakjnflkajnlkjnkjn}}}\exp(-c_{\text{\tiny\ref{akjakjnflkajnlkjnkjn}}}
\kappa_{\text{\tiny\ref{aljhvflajhfblajhbljh}}}^2 (|J|-1))\Big)^{d-|J|+1}
+\exp(-4|J|)\bigg)\\
&\cdot 2
\Big(\sup\big\{\Prob\{\langle X_1,y'\rangle\geq 0\}:\;\;y'\in S^{d-1},\;\LCD_{\alpha,\gamma_0}(y')
\geq \exp(c_{\text{\tiny\ref{aljhvflajhfblajhbljh}}}\, d)\big\}\Big)^{n-d},
\end{align*}
where $C_{\text{\tiny\ref{akjakjnflkajnlkjnkjn}}}=C_{\text{\tiny\ref{akjakjnflkajnlkjnkjn}}}(K,\gamma_0)$ and
$c_{\text{\tiny\ref{akjakjnflkajnlkjnkjn}}}=c_{\text{\tiny\ref{akjakjnflkajnlkjnkjn}}}(K,\gamma_0)$.
Since the distribution of $X_1$ is symmetric, we have
$$
2\,\Prob\{\langle X_1,y'\rangle\geq 0\}=1+\Prob\{\langle X_1,y'\rangle= 0\},
$$
where, by Theorem~\ref{akjakjnflkajnlkjnkjn}, for every $y'$ with $\LCD_{\alpha,\gamma_0}(y')
\geq \exp(c_{\text{\tiny\ref{aljhvflajhfblajhbljh}}}\, d)$,
$$
\Prob\{\langle X_1,y'\rangle= 0\}\leq
\frac{C_{\text{\tiny\ref{akjakjnflkajnlkjnkjn}}}}{\exp(c_{\text{\tiny\ref{aljhvflajhfblajhbljh}}}\, d)}
+C_{\text{\tiny\ref{akjakjnflkajnlkjnkjn}}}
\exp(-c_{\text{\tiny\ref{akjakjnflkajnlkjnkjn}}}\alpha^2).
$$
Combining all the estimates together, we get
$$
\Prob\big(\Event_{J,I}\setminus (\Event_4\cup\Event_5)\big)\leq C\,2^{-|J|-n+d}{d\choose |J|-1}
\exp\big(-\min(c_{\text{\tiny\ref{akjakjnflkajnlkjnkjn}}}
\kappa_{\text{\tiny\ref{aljhvflajhfblajhbljh}}}^2/2, c_{\text{\tiny\ref{aljhvflajhfblajhbljh}}},4)\,|J|\big),
$$
where $C>0$ depends on $K$.
The last estimate is sufficiently strong in the regime when $|J|$ is close to $d$.
Specifically, denoting $c:=\min(c_{\text{\tiny\ref{akjakjnflkajnlkjnkjn}}}
\kappa_{\text{\tiny\ref{aljhvflajhfblajhbljh}}}^2/2, c_{\text{\tiny\ref{aljhvflajhfblajhbljh}}},4)$
and assuming that an integer $\ell\in[\sqrt{d},d]$ satisfies
\begin{equation}\label{apkjfnfbfygviy}
\bigg(\frac{e\sqrt{2}d}{d-(\ell-1)}\bigg)^{2d-2\ell}\leq \exp(c\ell/2),
\end{equation}
we have, in view of the above,
\begin{align*}
\sum_{J\subset I\subset[n+1]:\,|J|=\ell,\;|I|=d}\Prob(\Event_{J,I}\setminus (\Event_4\cup\Event_5))
&\leq {n+1\choose d}{d\choose \ell}\,C\,2^{-\ell-n+d}{d\choose \ell-1}\exp(-c\ell)\\
&\leq C\,\exp(-c\ell)\bigg(\frac{e\sqrt{2}d}{d-(\ell-1)}\bigg)^{2d-2\ell+2}\\
&\leq C\,2e^2d^2\,\exp(-c\ell/2).
\end{align*}
For $\ell$ not satisfying \eqref{apkjfnfbfygviy}, we shall apply Lemma~\ref{ajhgcvutyrwvutqwrc}.
Again, let $i_{\max}$ be the largest index of $J$.
Observe that
\begin{align*}
\Event_{J,I}\setminus (\Event_4\cup\Event_5)\subset\big\{
&\mbox{$\|X_{i_{\max}}\|_2\in[\sqrt{d}/2,2\sqrt{d}]$,}\\
&\mbox{and
there is a $(\delta,\rho)$--incompressible vector $\lambda$ satisfying $W_J\lambda={\bf 0}$}
\big\},
\end{align*}
and that, whenever $\ell$ does not satisfy \eqref{apkjfnfbfygviy}, necessarily
$$
d/\ell\geq 1+s
$$
for some $s>0$ depending only on $K$. Applying Lemma~\ref{ajhgcvutyrwvutqwrc} with this
choice of $s$ and with $M=5$, we obtain
$$
\Prob\big(\Event_{J,I}\setminus (\Event_4\cup\Event_5)\big)\leq
\exp(-5d),
$$
and the union bound argument gives
$$
\sum_{J\subset I\subset[n+1]:\,|J|=\ell,\;|I|=d}\Prob(\Event_{J,I}\setminus (\Event_4\cup\Event_5))
\leq \,2^d\,2^{n+1}\,\exp(-5d)<\exp(-d).
$$
Finally, combining the two bounds for ``large'' and ``small'' $\ell$, we get
$$
\sum_{J\subset I\subset[n+1]:\,\sqrt{d}\leq |J|\leq d,\;|I|=d}\Prob(\Event_{J,I}\setminus (\Event_4\cup\Event_5))
\leq C''\,d^3\exp(-c''d)
$$
for some $c'',C''>0$ depending only on $K$.
Collecting together all the bounds from the proof, we obtain
$$
\Prob\big\{\mbox{$P_d'$ contains the origin on its boundary}\big\}\leq \tilde C\,d^3\exp(-\tilde cd)
$$
for some $\tilde c,\tilde C>0$ depending only on $K$.
By redefining the constants once more, we get the result.
\end{proof}

\section{Further questions}

\begin{Problem*}[Convex hull of sparse random vectors]
Let $X_1,\dots,X_n$ be random vectors with i.i.d components equidistributed with a product $b\xi$,
where $\xi$ is a symmetric subgaussian variable of unit variance, $b$ is Bernoulli($p$), and $\xi$ and $b$
are independent.
\begin{itemize}
\item We expect that for $p=\omega(\log d/d)$,
$$\Prob\big\{\conv\{X_1,\dots,X_n\}\mbox{ does not contain the origin}\big\}
=2^{-n+1}\sum_{k=0}^{d-1}{n-1\choose k}-o(1).$$
\item Assuming the above is true,
it is of theoretical interest to determine the critical sparsity
at which the above formula breaks. It seems plausible that, up to a $1\pm o(1)$
multiple, the critical value is a solution of the equation $(1-p)^{2d}=1/d$ i.e roughly
corresponds to the value of $p$ at which the $2d\times d$ matrix with i.i.d entries
equidistributed with $b\xi$ has a zero column with a constant (non-vanishing and not approaching one) probability.
\end{itemize}
\end{Problem*}

\begin{Problem*}[Non-symmetric distributions]
It was shown by Wagner\ and\ Welzl \cite{WW01} that, for i.i.d random vectors $X_1,\dots,X_n$ in $\R^d$
with absolutely continuous distributions, 
$$
\Prob\big\{\conv\{X_1,\dots,X_n\}\mbox{ contains the origin}\big\}\leq
1-2^{-n+1}\sum_{k=0}^{d-1}{n-1\choose k}.
$$
It may be expected that under the additional assumption that the coordinates of $X_i$'s
are i.i.d mean zero random variables, the inequality can be reversed up to losing an $o(1)$ term:
$$
\Prob\big\{\conv\{X_1,\dots,X_n\}\mbox{ contains the origin}\big\}\geq
1-2^{-n+1}\sum_{k=0}^{d-1}{n-1\choose k}-o(1).
$$
\end{Problem*}


\begin{thebibliography}{99}

\bibitem{B82}
{
K.-H. Borgwardt, The average number of pivot steps required by the simplex-method is polynomial, Z. Oper. Res. Ser. A-B {\bf 26} (1982), no.~5, {\rm A}157--{\rm A}177. MR0686603
}

\bibitem{B87}
{
K.-H. Borgwardt, {\it The simplex method}, Algorithms and Combinatorics: Study and Research Texts, 1, Springer, Berlin, 1987. MR0868467
}

\bibitem{B99}
{
K.-H. Borgwardt, A sharp upper bound for the expected number of shadow vertices in LP-polyhedra under orthogonal projection on two-dimensional planes, Math. Oper. Res. {\bf 24} (1999), no.~3, 544--603. MR1854243
}

\bibitem{DH18}
{
D. Dadush and S. Huiberts. 2018. A friendly smoothed analysis of the simplex method. In Proceedings of the 50th Annual ACM SIGACT Symposium on Theory of Computing (STOC 2018). Association for Computing Machinery, New York, NY, USA, 390–403. https://doi.org/10.1145/3188745.3188826
}

\bibitem{Eldan}
{
R. Eldan, Extremal points of high-dimensional random walks and mixing times of a Brownian motion on the sphere, Ann. Inst. Henri Poincar\'{e} Probab. Stat. {\bf 50} (2014), no.~1, 95--110. MR3161524
}

\bibitem{GKZ21}
{
F. G\"{o}tze, Z. Kabluchko\ and\ D. Zaporozhets, Grassmann angles and absorption probabilities of Gaussian convex hulls, Zap. Nauchn. Sem. S.-Peterburg. Otdel. Mat. Inst. Steklov. (POMI) {\bf 501} (2021), Veroyatnost i Statistika. 30, 126--148. MR4328185
}

\bibitem{HLO23}
{
S. Hayakawa, T.~J. Lyons\ and\ H. Oberhauser, Estimating the probability that a given vector is in the convex hull of a random sample, Probab. Theory Related Fields {\bf 185} (2023), no.~3-4, 705--746. MR4556280
}

\bibitem{HLZ22}
{
S. Huiberts, Y. T. Lee, X. Zhang,
Upper and Lower Bounds on the Smoothed Complexity of the Simplex Method,
arXiv:2211.11860
}

\bibitem{KZ20}
{
Z.~A. Kabluchko\ and\ D.~N. Zaporozhets, Absorption probabilities for Gaussian polytopes and regular spherical simplices, Adv. in Appl. Probab. {\bf 52} (2020), no.~2, 588--616. MR4123647
}

\bibitem{KVZGAFA}
{
Z.~A. Kabluchko, V.~V. Vysotski\u{\i}\ and\ D.~N. Zaporozhets, Convex hulls of random walks, hyperplane arrangements, and Weyl chambers, Geom. Funct. Anal. {\bf 27} (2017), no.~4, 880--918. MR3678504
}

\bibitem{KVZ17}
{
Z.~A. Kabluchko, V.~V. Vysotski\u{\i}\ and\ D.~N. Zaporozhets, Convex hulls of random walks: expected number of faces and face probabilities, Adv. Math. {\bf 320} (2017), 595--629. MR3709116
}

\bibitem{LPRT05}
{
A.~E. Litvak, A. Pajor, M. Rudelson, N. Tomczak-Jaegermann,
Smallest singular value of random matrices and geometry of random polytopes, Adv. Math. {\bf 195} (2005), no.~2, 491--523. MR2146352
}

\bibitem{PZ10}
{
G. Pan\ and\ W. Zhou, Circular law, extreme singular values and potential theory, J. Multivariate Anal. {\bf 101} (2010), no.~3, 645--656. MR2575411
}

\bibitem{RV08}
{
M. Rudelson\ and\ R. Vershynin, The Littlewood-Offord problem and invertibility of random matrices, Adv. Math. {\bf 218} (2008), no.~2, 600--633. MR2407948
}

\bibitem{RV09}
{
M. Rudelson\ and\ R. Vershynin, Smallest singular value of a random rectangular matrix, Comm. Pure Appl. Math. {\bf 62} (2009), no.~12, 1707--1739. MR2569075
}

\bibitem{ST04}
{
D.~A. Spielman\ and\ S.-H. Teng, Smoothed analysis of algorithms: why the simplex algorithm usually takes polynomial time, J. ACM {\bf 51} (2004), no.~3, 385--463. MR2145860
}

\bibitem{TY17}
{
K.~E. Tikhomirov\ and\ P. Youssef, When does a discrete-time random walk in $\Bbb{R}^n$ absorb the origin into its convex hull?, Ann. Probab. {\bf 45} (2017), no.~2, 965--1002. MR3630291
}

\bibitem{V09}
{
R. Vershynin, Beyond Hirsch conjecture: walks on random polytopes and smoothed complexity of the simplex method, SIAM J. Comput. {\bf 39} (2009), no.~2, 646--678. MR2529774
}

\bibitem{VershyninBook}
{
R. Vershynin, {\it High-dimensional probability}, Cambridge Series in Statistical and Probabilistic Mathematics, 47, Cambridge Univ. Press, Cambridge, 2018. MR3837109
}

\bibitem{VZ18}
{
V.~V. Vysotski\u{\i}\ and\ D.~N. Zaporozhets, Convex hulls of multidimensional random walks, Trans. Amer. Math. Soc. {\bf 370} (2018), no.~11, 7985--8012. MR3852455
}

\bibitem{WW01}
{
U. Wagner\ and\ E. Welzl, A continuous analogue of the upper bound theorem, Discrete Comput. Geom. {\bf 26} (2001), no.~2, 205--219. MR1843437
}

\bibitem{Wendel}
{
J.~G. Wendel, A problem in geometric probability, Math. Scand. {\bf 11} (1962), 109--111. MR0146858
}

\end{thebibliography}
\end{document}